\documentclass[10.6pt]{amsart}%
\usepackage{amssymb,amsmath,amsfonts,latexsym,amsthm,geometry}
\usepackage{amsmath}
\usepackage{amsfonts}
\usepackage{amssymb}
\usepackage{graphicx}
\usepackage{color}
\usepackage[colorlinks=true,linkcolor=blue,citecolor=blue]{hyperref}%
\providecommand{\U}[1]{\protect\rule{.1in}{.1in}}
\providecommand{\U}[1]{\protect\rule{.1in}{.1in}}
\providecommand{\U}[1]{\protect\rule{.1in}{.1in}}
\providecommand{\U}[1]{\protect\rule{.1in}{.1in}}
\newtheorem{theorem}{Theorem}[section]

\theoremstyle{definition}

\newtheorem{remark}[theorem]{Remark}

\begin{document}
\title[The Orlicz inequality for multilinear forms]{The Orlicz inequality for multilinear forms}
\author[D. N\'{u}\~{n}ez]{Daniel N\'{u}\~{n}ez-Alarc\'{o}n}
\address{Departamento de Matem\'{a}ticas\\
\indent Universidad Nacional de Colombia\\
\indent111321 - Bogot\'{a}, Colombia}
\email{dnuneza@unal.edu.co}
\author[D. Pellegrino]{Daniel Pellegrino}
\address{Departamento de Matem\'{a}tica \\
Universidade Federal da Para\'{\i}ba \\
58.051-900 - Jo\~{a}o Pessoa, Brazil.}
\email{daniel.pellegrino@academico.ufpb.br}
\author[D. Serrano]{Diana Serrano-Rodr\'{\i}guez}
\address{Departamento de Matem\'{a}ticas\\
\indent Universidad Nacional de Colombia\\
\indent111321 - Bogot\'{a}, Colombia}
\email{diserranor@unal.edu.co}
\thanks{D. Pellegrino is supported by CNPq Grant 307327/2017-5 and Grant 2019/0014
Paraiba State Research Foundation (FAPESQ) }

\begin{abstract}
The Orlicz $\left(  \ell_{2},\ell_{1}\right)  $-mixed inequality states that
\[
\left(  \sum_{j_{1}=1}^{n}\left(  \sum_{j_{2}=1}^{n}\left\vert A(e_{j_{1}%
},e_{j_{2}})\right\vert \right)  ^{2}\right)  ^{\frac{1}{2}}\leq\sqrt
{2}\left\Vert A\right\Vert
\]
for all bilinear forms $A:\mathbb{K}^{n}\times\mathbb{K}^{n}\rightarrow
\mathbb{K}$ and all positive integers $n$, where $\mathbb{K}^{n}$ denotes
$\mathbb{R}^{n}$ or $\mathbb{C}^{n}$ endowed with the supremum norm. In this
paper we extend this inequality to multilinear forms, with $\mathbb{K}^{n}$
endowed with $\ell_{p}$ norms for all $p\in\lbrack1,\infty].$

\end{abstract}
\maketitle

\section{Introduction}

The origins of the theory of summability of multilinear forms and absolutely
summing multilinear operators are probably associated to Orlicz $\left(
\ell_{2},\ell_{1}\right)  $-mixed inequality published in the $1930$'s (see
\cite[page 24]{blei}). It states that
\[
\left(  \sum_{j_{1}=1}^{n}\left(  \sum_{j_{2}=1}^{n}\left\vert A(e_{j_{1}%
},e_{j_{2}})\right\vert \right)  ^{2}\right)  ^{\frac{1}{2}}\leq\sqrt
{2}\left\Vert A\right\Vert
\]
for all bilinear forms $A:\mathbb{K}^{n}\times\mathbb{K}^{n}\rightarrow
\mathbb{K}$, and all positive integers $n.$ Here and henceforth $\mathbb{K}%
=\mathbb{R}$ or $\mathbb{C}$ and $\mathbb{K}^{n}$ is endowed with the supremum
norm. We also represent by $e_{k}$ the canonical vectors in a sequence space
and
\[
\left\Vert A\right\Vert :=\sup\left\{  \left\vert A(x,y)\right\vert
:\left\Vert x\right\Vert \leq1\text{ and }\left\Vert y\right\Vert
\leq1\right\}  .
\]
An equivalent formulation is the following:
\begin{equation}
\left(  \sum_{j_{1}=1}^{\infty}\left(  \sum_{j_{2}=1}^{\infty}\left\vert
A(e_{j_{1}},e_{j_{2}})\right\vert \right)  ^{2}\right)  ^{\frac{1}{2}}%
\leq\sqrt{2}\left\Vert A\right\Vert \label{orlicz}%
\end{equation}
for all continuous bilinear forms $A:c_{0}\times c_{0}\rightarrow\mathbb{K}$.
The exponents in (\ref{orlicz}) are optimal in the sense that, fixing the
exponent $1$, the exponent $2$ cannot be replaced by smaller exponents (nor
the exponent $1$ can be replaced by smaller exponents) keeping the constant
independent of $n$. The Orlicz inequality is closely related to Littlewood's
$\left(  \ell_{1},\ell_{2}\right)  $-mixed inequality (see \cite[page
23]{blei}), which asserts that
\[
\sum_{j_{1}=1}^{\infty}\left(  \sum_{j_{2}=1}^{\infty}\left\vert A(e_{j_{1}%
},e_{j_{2}})\right\vert ^{2}\right)  ^{\frac{1}{2}}\leq\sqrt{2}\left\Vert
A\right\Vert
\]
for all continuous bilinear forms $A:c_{0}\times c_{0}\rightarrow\mathbb{K}$.
Again, the exponents are optimal in the same sense above described. Combining
these two inequalities, and using the H\"{o}lder inequality for mixed sums we
recover Littlewood's $4/3$ inequality:
\[
\left(  \sum_{j_{1},j_{2}=1}^{\infty}\left\vert A(e_{j_{1}},e_{j_{2}%
})\right\vert ^{\frac{4}{3}}\right)  ^{\frac{3}{4}}\leq\sqrt{2}\left\Vert
A\right\Vert
\]
for all continuous bilinear forms $A:c_{0}\times c_{0}\rightarrow\mathbb{K}$.
For recent results on absolutely summing linear and multilinear operators we
refer the interested reader to \cite{bayart5, mastilo, rueda} and the
references therein.

The exponent $4/3$ from the previous inequality cannot be replaced by smaller
exponents keeping the constant independent of $n$. The constant $\sqrt{2}$ is
optimal (in all the three inequalities) when $\mathbb{K}=\mathbb{R}$, but the
optimal constants when $\mathbb{K}=\mathbb{C}$ are unknown.

In 1934 Hardy and Littlewood \cite{hardy} (see also \cite{tonge}) pushed the
subject further, extending the above results to bilinear forms defined on
$\ell_{p}$ spaces (when $p=\infty$ we consider $c_{0}$ instead of
$\ell_{\infty}$). The investigation of extensions of the Hardy--Littlewood
inequalities to multilinear forms were initiated by Praciano-Pereira
\cite{pra} in 1981 and intensively investigated since then (see, for instance,
\cite{abps2, rezende, laaaron, blasco, mastilo2, tonge, paulino, psst}), but
there are still several open problems regarding the optimal exponents and
optimal constants involved.

For the sake of simplicity, we shall use the same notation from \cite{abps2}:%
\[
X_{p}:=\left\{
\begin{array}
[c]{c}%
\ell_{p}\text{, if }p\in\lbrack1,\infty)\\
c_{0}\text{, if }p=\infty
\end{array}
\right.
\]
and, when $q=\infty$, the sum $\left(  \sum_{j}\left\Vert x_{j}\right\Vert
^{q}\right)  ^{1/q}$ shall represent the supremum of $\left\Vert
x_{j}\right\Vert $. We also assume that $1/0=\infty$ and $1/\infty=0$ and
denote the conjugate index of $s$ by $s^{\ast},$ i.e., $1/s+1/s^{\ast}=1$. One
of the main goals of this line of research is to find the optimal values of
the exponents $s_{1},...,s_{m}$ and of the constants $C_{s_{1},...,s_{m}%
}^{\left(  \mathbb{K}\right)  \text{ }p_{1},...,p_{m}}$ satisfying%
\[
\left(  \sum_{j_{1}=1}^{\infty}...\left(  \sum_{j_{m-1}=1}^{\infty}\left(
\sum_{j_{m}=1}^{\infty}\left\vert A(e_{j_{1}},...e_{j_{m}})\right\vert
^{s_{m}}\right)  ^{\frac{s_{m-1}}{s_{m}}}\right)  ^{\frac{s_{m-2}}{s_{m-1}}%
}...\right)  ^{\frac{1}{s_{1}}}\leq C_{s_{1},...,s_{m}}^{\left(
\mathbb{K}\right)  \text{ }p_{1},...,p_{m}}\left\Vert A\right\Vert
\]
for all continuous $m$-linear forms $A:X_{p_{1}}\times\cdots\times X_{p_{m}%
}\rightarrow\mathbb{K}$. The answer is known in several cases (see
\cite{abps2, laaaron, sv} and the references therein), but a complete solution
is still unknown. In this note we shall be interested in investigating the
optimal exponents $s_{1},...,s_{m}.$ It is simple to prove that the optimal
exponent $s_{m}$ associated to the sum $\sum_{j_{m}=1}^{\infty}$ is
$p_{m}^{\ast}$. Our main result provides the optimal exponents $s_{1}%
,...,s_{m-1}$ in the case that $s_{m}=p_{m}^{\ast}.$

From now on, let $r\geq2,$ and let $s_{1},...,s_{m}\,\in\lbrack1,\infty]$. Let
us define $\delta^{s_{k},...,s_{m}}$ and $\lambda_{r}^{s_{k},...,s_{m}}$ by
\[
\delta^{s_{k},...,s_{m}}:=\frac{1}{\max\left\{  1-\left(  \frac{1}{s_{k}%
}+\cdots+\frac{1}{s_{m}}\right)  ,0\right\}  },
\]
and
\[
\lambda_{r}^{s_{k},...,s_{m}}:=\frac{1}{\max\left\{  \frac{1}{r}-\left(
\frac{1}{s_{k}}+\cdots+\frac{1}{s_{m}}\right)  ,0\right\}  },
\]
for all positive integers $m$ and $k=1,...,m$. Note that when $1/s_{k}%
+\cdots+1/s_{m}\geq1$ we have%
\[
\delta^{s_{k},...,s_{m}}=\infty
\]
and, also, when $1/s_{k}+\cdots+1/s_{m}\geq\frac{1}{r}$ we have%
\[
\lambda_{r}^{s_{k},...,s_{m}}=\infty.
\]
Our main result is, in some sense, a generalization of the the Orlicz
inequality. In fact, if we consider the very particular case $\left(
m,p_{1},p_{2}\right)  =\left(  2,\infty,\infty\right)  $ and $\sigma$ as the
identity map in its statement, we recover the Orlicz inequality:

\begin{theorem}
\label{orl}Let $m\geq2$ be an integer and $\sigma:\{1,...,m\}\rightarrow
\{1,...,m\}$ be a bijection. If
\begin{align*}
\left(  q_{1},...,q_{m-1}\right)   &  \in(0,\infty]^{m-1},\\
\left(  p_{1},...,p_{m}\right)   &  \in\lbrack1,\infty]^{m},
\end{align*}
the following assertions are equivalent:

(1) There is a constant $C_{p_{1},...,p_{m}}\geq1$ such that%
\[
\left(  \sum_{j_{\sigma(1)}=1}^{\infty}\left(  \sum_{j_{\sigma(2)}=1}^{\infty
}\cdots\left(  \sum_{j_{\sigma(m)}=1}^{\infty}\left\vert A(e_{j_{\sigma(1)}%
},...,e_{j_{\sigma(m)}})\right\vert ^{p_{\sigma(m)}^{\ast}}\right)
^{\frac{q_{m-1}}{p_{\sigma(m)}^{\ast}}}\cdots\right)  ^{\frac{q_{1}}{q_{2}}%
}\right)  ^{\frac{1}{q_{1}}}\leq C_{p_{1},...,p_{m}}\left\Vert A\right\Vert
\]
for all continuous $m$-linear forms $A:X_{p_{1}}\times\cdots\times X_{p_{m}%
}\rightarrow\mathbb{K}$.

(2) The exponents $q_{1},...,q_{m-1}$ satisfy%
\[
q_{1}\geq\delta^{p_{\sigma(1)},...,p_{\sigma(m-1)},\mu},q_{2}\geq
\delta^{p_{\sigma(2)},...,p_{\sigma(m-1)},\mu},...,q_{m-1}\geq\delta
^{p_{\sigma(m-1)},\mu},
\]
where $\mu=\min\{p_{\sigma(m)},2\}.$
\end{theorem}

\section{Preliminary results}

Let $2\leq q<\infty$ and $0<s<\infty$. Recall that a Banach space $X$ has
cotype $q$ if there is a constant $C>0$ such that, no matter how we select
finitely many vectors $x_{1},\dots,x_{n}\in X$,%
\begin{equation}
\left(  \sum_{j=1}^{n}\Vert x_{j}\Vert^{q}\right)  ^{\frac{1}{q}}\leq C\left(
\int_{[0,1]}\left\Vert \sum_{j=1}^{n}r_{j}(t)x_{j}\right\Vert ^{2}dt\right)
^{1/2}, \label{99}%
\end{equation}
where $r_{j}$ denotes the $j$-th Rademacher function. The infimum of the
cotypes of $X$ is denoted by $\cot X$.

The following result was proved in \cite{laaaron}:

\begin{theorem}
(see \cite{laaaron}) \label{661}\ Let $\left(  q_{1},...,q_{m}\right)
\in(0,\infty)^{m}$, and $Y$ be an infinite-dimensional Banach space with
cotype $\cot Y$. If
\begin{equation}
\frac{1}{p_{1}}+\cdots+\frac{1}{p_{m}}<\frac{1}{\cot Y}, \label{cond}%
\end{equation}
then the following assertions are equivalent:

(a) There is a constant $C_{p_{1},...,p_{m}}^{Y}\geq1$ such that%
\[
\left(  \sum_{j_{1}=1}^{\infty}\left(  \sum_{j_{2}=1}^{\infty}\cdots\left(
\sum_{j_{m}=1}^{\infty}\left\Vert A(e_{j_{1}},...,e_{j_{m}})\right\Vert
^{q_{m}}\right)  ^{\frac{q_{m-1}}{q_{m}}}\cdots\right)  ^{\frac{q_{1}}{q_{2}}%
}\right)  ^{\frac{1}{q_{1}}}\leq C_{p_{1},...,p_{m}}^{Y}\left\Vert
A\right\Vert
\]
for all continuous $m$-linear operators $A:X_{p_{1}}\times\cdots\times
X_{p_{m}}\rightarrow Y.$

(b) The exponents $q_{1},...,q_{m}$ satisfy
\[
q_{1}\geq\lambda_{\cot Y}^{p_{1},...,p_{m}},q_{2}\geq\lambda_{\cot Y}%
^{p_{2},...,p_{m}},...,q_{m-1}\geq\lambda_{\cot Y}^{p_{m-1},p_{m}},q_{m}%
\geq\lambda_{\cot Y}^{p_{m}}.
\]

\end{theorem}

We need the following extension of the previous theorem, relaxing the
hypothesis (\ref{cond}). Besides, below we have $\left(  q_{1},...,q_{m}%
\right)  \in(0,\infty]^{m}$ while in Theorem \ref{661} we have $\left(
q_{1},...,q_{m}\right)  \in(0,\infty)^{m}$.

\begin{theorem}
\label{cotcrit}\ Let $\left(  q_{1},...,q_{m}\right)  \in(0,\infty]^{m},$
$\left(  p_{1},...,p_{m}\right)  \,\in\lbrack1,\infty]^{m}$ and $Y$ be an
infinite-dimensional Banach space with cotype $\cot Y.$ The following
assertions are equivalent:

(a) There is a constant $C_{p_{1},...,p_{m}}^{Y}\geq1$ such that%
\begin{equation}
\left(  \sum_{j_{1}=1}^{\infty}\left(  \sum_{j_{2}=1}^{\infty}\cdots\left(
\sum_{j_{m}=1}^{\infty}\left\Vert A(e_{j_{1}},...,e_{j_{m}})\right\Vert
^{q_{m}}\right)  ^{\frac{q_{m-1}}{q_{m}}}\cdots\right)  ^{\frac{q_{1}}{q_{2}}%
}\right)  ^{\frac{1}{q_{1}}}\leq C_{p_{1},...,p_{m}}^{Y}\left\Vert
A\right\Vert \label{yy}%
\end{equation}
for all continuous $m$-linear operators $A:X_{p_{1}}\times\cdots\times
X_{p_{m}}\rightarrow Y.$

(b) The exponents $q_{1},...,q_{m}$ satisfy
\[
q_{1}\geq\lambda_{\cot Y}^{p_{1},...,p_{m}},q_{2}\geq\lambda_{\cot Y}%
^{p_{2},...,p_{m}},\ldots,q_{m-1}\geq\lambda_{\cot Y}^{p_{m-1},p_{m}}%
,q_{m}\geq\lambda_{\cot Y}^{p_{m}}.
\]

\end{theorem}

\begin{proof}
We begin by proving the direct implication. We just need to consider the
case\textbf{\ }%
\begin{equation}
\frac{1}{p_{1}}+\cdots+\frac{1}{p_{m}}\geq\frac{1}{\cot Y},\label{maior}%
\end{equation}
since the other case is covered by Theorem \ref{661}. By the Maurey--Pisier
factorization result (see \cite[pages 286, 287]{Di}), the Banach space $Y$
finitely factors the formal inclusion $\ell_{\cot Y}\hookrightarrow
\ell_{\infty}$, i.e., there are universal constants $C_{1},C_{2}>0$ such that,
for all $n$, there are vectors $z_{1},...,z_{n}\in Y$ satisfying%
\begin{equation}
C_{1}\left\Vert \left(  a_{j}\right)  _{j=1}^{n}\right\Vert _{\infty}%
\leq\left\Vert \sum\limits_{j=1}^{n}a_{j}z_{j}\right\Vert \leq C_{2}\left(
\sum\limits_{j=1}^{n}\left\vert a_{j}\right\vert ^{\cot Y}\right)  ^{1/\cot
Y},\label{pisi}%
\end{equation}
for all sequences of scalars $\left(  a_{j}\right)  _{j=1}^{n}.$ Consider the
continuous $m$-linear operator $A_{n}\colon X_{p_{1}}\times\cdots\times
X_{p_{m}}\rightarrow Y$ given by%
\begin{equation}
A_{n}(x^{(1)},\ldots,x^{(m)})=\sum\limits_{j=1}^{n}x_{j}^{\left(  1\right)
}x_{j}^{\left(  2\right)  }\cdots x_{j}^{\left(  m\right)  }z_{j}.\label{22aa}%
\end{equation}
By (\ref{pisi}) and the H\"{o}lder inequality we have%
\begin{align}
\Vert A_{n}\Vert &  =\sup_{\left\Vert x^{\left(  1\right)  }\right\Vert
_{p_{1}}\leq1,\ldots,\left\Vert x^{\left(  m\right)  }\right\Vert _{p_{m}}%
\leq1}\left\Vert \sum\limits_{j=1}^{n}x_{j}^{\left(  1\right)  }%
...x_{j}^{\left(  m\right)  }z_{j}\right\Vert \label{33aa}\\
&  \leq\sup_{\left\Vert x^{\left(  1\right)  }\right\Vert _{p_{1}}\leq
1,\ldots,\left\Vert x^{\left(  m\right)  }\right\Vert _{p_{m}}\leq1}%
C_{2}\left(  \sum\limits_{j=1}^{n}\left\vert x_{j}^{\left(  1\right)
}...x_{j}^{\left(  m\right)  }\right\vert ^{\cot Y}\right)  ^{1/\cot
Y}\nonumber\\
&  \leq\sup_{\left\Vert x^{\left(  1\right)  }\right\Vert _{p_{1}}\leq
1,\ldots,\left\Vert x^{\left(  m\right)  }\right\Vert _{p_{m}}\leq1}%
C_{2}\left(  \prod_{k=1}^{m}\left(  \sum\limits_{j=1}^{n}\left\vert
x_{j}^{\left(  k\right)  }\right\vert ^{p_{k}}\right)  ^{1/p_{k}}\right)
\nonumber\\
&  =C_{2}.\nonumber
\end{align}
Note that, by (\ref{22aa}), we have
\[
\left(  \sum_{j_{1}=1}^{n}\left(  \sum_{j_{2}=1}^{n}\cdots\left(  \sum
_{j_{m}=1}^{n}\left\Vert A_{n}(e_{j_{1}},...,e_{j_{m}})\right\Vert ^{q_{m}%
}\right)  ^{\frac{q_{m-1}}{q_{m}}}\cdots\right)  ^{\frac{q_{1}}{q_{2}}%
}\right)  ^{\frac{1}{q_{1}}}=\left(  \sum_{j=1}^{n}\left\Vert z_{j}\right\Vert
^{q_{1}}\right)  ^{\frac{1}{q_{1}}}.
\]
Thus, by (\ref{pisi}) we conclude that%
\[
\left(  \sum_{j_{1}=1}^{n}\left(  \sum_{j_{2}=1}^{n}\cdots\left(  \sum
_{j_{m}=1}^{n}\left\Vert A_{n}(e_{j_{1}},\ldots,e_{j_{m}})\right\Vert ^{q_{m}%
}\right)  ^{\frac{q_{m-1}}{q_{m}}}\cdots\right)  ^{\frac{q_{1}}{q_{2}}%
}\right)  ^{\frac{1}{q_{1}}}\geq C_{1}n^{\frac{1}{q_{1}}}.
\]
Combining the previous inequality with (\ref{yy}) and (\ref{33aa}) we conclude
that%
\[
C_{1}n^{1/q_{1}}\leq C_{p_{1},...,p_{m}}^{Y}C_{2}.
\]
Thus, since $n$ is arbitrary, we have%
\begin{equation}
q_{1}=\infty=\lambda_{\cot Y}^{p_{1},...,p_{m}}.\label{899}%
\end{equation}
If
\[
\frac{1}{p_{i}}+\cdots+\frac{1}{p_{m}}\geq\frac{1}{\cot Y}%
\]
for all $i$, the proof is immediate. Otherwise, let $i_{0}\in\{2,3,\ldots,m\}$
be the smallest index such that
\[
\left\{
\begin{array}
[c]{c}%
\frac{1}{p_{i_{0}}}+\cdots+\frac{1}{p_{m}}<\frac{1}{\cot Y},\\
\frac{1}{p_{i_{0}-1}}+\cdots+\frac{1}{p_{m}}\geq\frac{1}{\cot Y}.
\end{array}
\right.
\]
If $i_{0}=2,$ note that by (\ref{899}) we have%
\begin{equation}
\sup_{j_{1}}\left(  \sum_{j_{2}=1}^{\infty}\left(  \cdots\left(  \sum
_{j_{m}=1}^{\infty}\left\Vert A(e_{j_{1}},\ldots,e_{j_{m}})\right\Vert
^{q_{m}}\right)  ^{\frac{q_{m-1}}{q_{m}}}\cdots\right)  ^{\frac{q_{2}}{q_{3}}%
}\right)  ^{\frac{1}{q_{2}}}\leq C_{p_{1},...,p_{m}}^{Y}\left\Vert
A\right\Vert \label{dez}%
\end{equation}
for all continuous $m$-linear operators $A:X_{p_{1}}\times\cdots\times
X_{p_{m}}\rightarrow Y.$ From (\ref{dez}) it is simple to show that
\[
\left(  \sum_{j_{2}=1}^{\infty}\left(  \cdots\left(  \sum_{j_{m}=1}^{\infty
}\left\Vert A(e_{j_{2}},\ldots,e_{j_{m}})\right\Vert ^{q_{m}}\right)
^{\frac{q_{m-1}}{q_{m}}}\cdots\right)  ^{\frac{q_{2}}{q_{3}}}\right)
^{\frac{1}{q_{2}}}\leq C_{p_{1},...,p_{m}}^{Y}\left\Vert A\right\Vert ,
\]
for all continuous $\left(  m-1\right)  $-linear operators $A:X_{p_{2}}%
\times\cdots\times X_{p_{m}}\rightarrow Y$. Since%
\[
\frac{1}{p_{2}}+\cdots+\frac{1}{p_{m}}<\frac{1}{\cot Y},
\]
by Theorem \ref{661} we conclude that
\[
q_{2}\geq\lambda_{\cot Y}^{p_{2},...,p_{m}},q_{3}\geq\lambda_{\cot Y}%
^{p_{3},...,p_{m}},\ldots,q_{m-1}\geq\lambda_{\cot Y}^{p_{m-1},p_{m}}%
,q_{m}\geq\lambda_{\cot Y}^{p_{m}}.
\]

If $i_{0}=3$, we consider
\[
A(x^{(1)},\ldots,x^{(m)})=x_{1}^{(1)}\sum\limits_{j=1}^{n}x_{j}^{(2)}\cdots
x_{j}^{(m)}z_{j}%
\]
and we can imitate the previous arguments to conclude that%
\[
q_{2}=\infty=\lambda_{\cot Y}^{p_{2},...,p_{m}}.
\]
and hence%
\begin{equation}
\sup_{j_{1},j_{2}}\left(  \sum_{j_{3}=1}^{\infty}\left(  \cdots\left(
\sum_{j_{m}=1}^{\infty}\left\Vert A(e_{j_{1}},\ldots,e_{j_{m}})\right\Vert
^{q_{m}}\right)  ^{\frac{q_{m-1}}{q_{m}}}\cdots\right)  ^{\frac{q_{3}}{q_{4}}%
}\right)  ^{\frac{1}{q_{3}}}\leq C_{p_{1},...,p_{m}}^{Y}\left\Vert
A\right\Vert ,\label{555}%
\end{equation}
for all continuous $m$-linear operators $A:X_{p_{1}}\times\cdots\times
X_{p_{m}}\rightarrow Y.$ Again, it is plain that
\[
\left(  \sum_{j_{3}=1}^{\infty}\left(  \cdots\left(  \sum_{j_{m}=1}^{\infty
}\left\Vert A(e_{j_{3}},\ldots,e_{j_{m}})\right\Vert ^{q_{m}}\right)
^{\frac{q_{m-1}}{q_{m}}}\cdots\right)  ^{\frac{q_{i_{0}+1}}{q_{i_{0}}}%
}\right)  ^{\frac{1}{q_{i_{0}}}}\leq C_{p_{1},...,p_{m}}^{Y}\left\Vert
A\right\Vert ,
\]
for all continuous $\left(  m-2\right)  $-linear operators $A:X_{p_{3}}%
\times\cdots\times X_{p_{m}}\rightarrow Y$. Since%
\[
\frac{1}{p_{3}}+\cdots+\frac{1}{p_{m}}<\frac{1}{\cot Y},
\]
by Theorem \ref{661} we have%
\[
q_{3}\geq\lambda_{\cot Y}^{p_{3},...,p_{m}},q_{4}\geq\lambda_{\cot Y}%
^{p_{4},...,p_{m}},\ldots,q_{m-1}\geq\lambda_{\cot Y}^{p_{m-1},p_{m}}%
,q_{m}\geq\lambda_{\cot Y}^{p_{m}}.
\]
We conclude the proof in a similar fashion for $i_{0}=4,...,m.$

Now we prove the reverse implication. The case%
\[
\frac{1}{p_{1}}+\cdots+\frac{1}{p_{m}}<\frac{1}{\cot Y}%
\]
is encompassed by Theorem \ref{661}. So, we shall consider
\[
\frac{1}{p_{1}}+\cdots+\frac{1}{p_{m}}\geq\frac{1}{\cot Y}.
\]
If
\[
\frac{1}{p_{i}}+\cdots+\frac{1}{p_{m}}\geq\frac{1}{\cot Y}%
\]
for all $i$, the proof is immediate. Otherwise, let $i_{0}\in\{2,...,m\}$ be
the smallest index such that
\[
\left\{
\begin{array}
[c]{c}%
\frac{1}{p_{i_{0}}}+\cdots+\frac{1}{p_{m}}<\frac{1}{\cot Y},\\
\frac{1}{p_{i_{0}-1}}+\cdots+\frac{1}{p_{m}}\geq\frac{1}{\cot Y}.
\end{array}
\right.
\]
We need to prove that there is a constant $C_{p_{1},...,p_{m}}^{Y}\geq1$, such
that%
\[
\sup_{j_{1},\ldots,j_{i_{0}-1}}\left(  \sum_{j_{i_{0}}=1}^{\infty}\left(
\cdots\left(  \sum_{j_{m}=1}^{\infty}\left\Vert A(e_{j_{1}},\ldots,e_{j_{m}%
})\right\Vert ^{q_{m}}\right)  ^{\frac{q_{m-1}}{q_{m}}}\cdots\right)
^{\frac{q_{i_{0}}}{q_{i_{0}+1}}}\right)  ^{\frac{1}{q_{i_{0}}}}\leq
C_{p_{1},...,p_{m}}^{Y}\left\Vert A\right\Vert
\]
for
\[
q_{i_{0}}\geq\lambda_{\cot Y}^{p_{i_{0}},...,p_{m}},...,q_{m}\geq\lambda_{\cot
Y}^{p_{m}}.
\]
By Theorem \ref{661}, we know that for any fixed vectors $e_{j_{1}%
},...,e_{j_{i_{0}-1}}$, there is a constant $C_{p_{i_{0}},...,p_{m}}^{Y}\geq
1$, such that
\[
\left(  \sum_{j_{i_{0}}=1}^{\infty}\left(  \cdots\left(  \sum_{j_{m}%
=1}^{\infty}\left\Vert A(e_{j_{1}},...,e_{j_{m}})\right\Vert ^{\lambda_{\cot
Y}^{p_{m}}}\right)  ^{\frac{\lambda_{\cot Y}^{p_{m-1},p_{m}}}{\lambda_{\cot
Y}^{p_{m}}}}\cdots\right)  ^{\frac{\lambda_{\cot Y}^{p_{i_{0}},...,p_{m}}%
}{\lambda_{\cot Y}^{p_{i_{0}+1},...,p_{m}}}}\right)  ^{\frac{1}{\lambda_{\cot
Y}^{p_{i_{0}},...,p_{m}}}}\leq C_{p_{1},\ldots,,p_{m}}^{Y}\left\Vert
A\right\Vert
\]
for all continuous $m$-linear operators $A\colon X_{p_{1}}\times\cdots\times
X_{p_{m}}\rightarrow Y$. Then,
\[
\sup_{j_{1},\ldots,j_{i_{0}-1}}\left(  \sum_{j_{i_{0}}=1}^{\infty}\left(
\cdots\left(  \sum_{j_{m}=1}^{\infty}\left\Vert A(e_{j_{1}},\ldots,e_{j_{m}%
})\right\Vert ^{\lambda_{\cot Y}^{p_{m}}}\right)  ^{\frac{\lambda_{\cot
Y}^{p_{m-1},p_{m}}}{\lambda_{\cot Y}^{p_{m}}}}\cdots\right)  ^{\frac
{\lambda_{\cot Y}^{p_{i_{0}},...,p_{m}}}{\lambda_{\cot Y}^{p_{i_{0}%
+1},...,p_{m}}}}\right)  ^{\frac{1}{\lambda_{\cot Y}^{p_{i_{0}},...,p_{m}}}%
}\leq C_{p_{1},\ldots,,p_{m}}^{Y}\left\Vert A\right\Vert
\]
for all continuous $m$-linear operators $A\colon X_{p_{1}}\times\cdots\times
X_{p_{m}}\rightarrow Y$.

To conclude the proof we just need to remark that
\begin{align*}
&  \sup_{j_{1},\ldots,j_{i_{0}-1}}\left(  \sum_{j_{i_{0}}=1}^{\infty}\left(
\cdots\left(  \sum_{j_{m}=1}^{\infty}\left\Vert A(e_{j_{1}},\ldots,e_{j_{m}%
})\right\Vert ^{q_{m}}\right)  ^{\frac{q_{m-1}}{q_{m}}}\cdots\right)
^{\frac{q_{i_{0}}}{q_{i_{0}+1}}}\right)  ^{\frac{1}{q_{i_{0}}}}\\
&  \leq\sup_{j_{1},\ldots,j_{i_{0}-1}}\left(  \sum_{j_{i_{0}}=1}^{\infty
}\left(  \cdots\left(  \sum_{j_{m}=1}^{\infty}\left\Vert A(e_{j_{1}}%
,\ldots,e_{j_{m}})\right\Vert ^{\lambda_{\cot Y}^{p_{m}}}\right)
^{\frac{\lambda_{\cot Y}^{p_{m-1},p_{m}}}{\lambda_{\cot Y}^{p_{m}}}}%
\cdots\right)  ^{\frac{\lambda_{\cot Y}^{p_{i_{0}},...,p_{m}}}{\lambda_{\cot
Y}^{p_{i_{0}+1},...,p_{m}}}}\right)  ^{\frac{1}{\lambda_{\cot Y}^{p_{i_{0}%
},...,p_{m}}}}%
\end{align*}
provided%
\[
q_{i_{0}}\geq\lambda_{\cot Y}^{p_{i_{0}},...,p_{m}},...,q_{m}\geq\lambda_{\cot
Y}^{p_{m}}.
\]

\end{proof}

\section{Proof of Theorem \ref{orl}}

Throughout this proof the adjoint of a Banach space $X$ will be denoted by
$X^{\ast}.$ To simplify the notation we will consider $\sigma(j)=j$ for all
$j;$ the other cases are similar. Let $\mathcal{L}^{m}\left(  X_{p_{1}}%
,\ldots,X_{p_{m}};Y\right)  $ denote the space of all continuous $m$-linear
operators from $X_{p_{1}}\times\cdots\times X_{p_{m}}$ to $Y.$ By the
canonical isometric isomorphism%
\[
\Psi:\mathcal{L}^{m}\left(  X_{p_{1}},,X_{p_{m}};\mathbb{K}\right)
\rightarrow\mathcal{L}^{m-1}\left(  X_{p_{1}},\dots,X_{p_{m-1}};\left(
X_{p_{m}}\right)  ^{\ast}\right)
\]
and duality in $X_{p_{m}}$, note that, if $R\in\mathcal{L}^{m}\left(
X_{p_{1}},\ldots,X_{p_{m}};\mathbb{K}\right)  $, we have%
\begin{equation}
R\left(  x_{1},...,x_{m-1},e_{n}\right)  =\Psi(R)\left(  x_{1},...,x_{m-1}%
\right)  (e_{n})=\left(  \Psi(R)\left(  x_{1},...,x_{m-1}\right)  \right)
_{n}.\label{900}%
\end{equation}
We start off by proving (1)$\Rightarrow$(2). Let us suppose that there is a
constant $C_{p_{1},\ldots,p_{m}}\geq1$ such that%
\begin{equation}
\left(  \sum_{j_{1}=1}^{\infty}\left(  \sum_{j_{2}=1}^{\infty}\cdots\left(
\sum_{j_{m}=1}^{\infty}\left\vert T(e_{j_{1}},...,e_{j_{m}})\right\vert
^{p_{m}^{\ast}}\right)  ^{\frac{q_{m-1}}{p_{m}^{\ast}}}\cdots\right)
^{\frac{q_{1}}{q_{2}}}\right)  ^{\frac{1}{q_{1}}}\leq C_{p_{1},...,p_{m}%
}\left\Vert T\right\Vert \label{1111}%
\end{equation}
for all continuous $m$-linear forms $T:X_{p_{1}}\times\cdots\times X_{p_{m}%
}\rightarrow\mathbb{K}$.

Consider a continuous $\left(  m-1\right)  $-linear operator $A:X_{p_{1}%
}\times\cdots\times X_{p_{m-1}}\rightarrow\left(  X_{p_{m}}\right)  ^{\ast}.$
Then, using our hypothesis, we have
\begin{align}
&  \left(  \sum_{j_{1}=1}^{\infty}\left(  \sum_{j_{2}=1}^{\infty}\cdots\left(
\sum_{j_{m-1}=1}^{\infty}\left\Vert A(e_{j_{1}},\ldots,e_{j_{m-1}})\right\Vert
_{\left(  X_{p_{m}}\right)  ^{\ast}}^{q_{m-1}}\right)  ^{\frac{q_{m-2}%
}{q_{m-1}}}\cdots\right)  ^{\frac{q_{1}}{q_{2}}}\right)  ^{\frac{1}{q_{1}}%
}\label{expos}\\
&  =\left(  \sum_{j_{1}=1}^{\infty}\left(  \sum_{j_{2}=1}^{\infty}%
\cdots\left(  \sum_{j_{m-1}=1}^{\infty}\left(  \sum_{j_{m}=1}^{\infty
}\left\vert \left(  A\left(  e_{j_{1}},\ldots,e_{j_{m-1}}\right)  \right)
_{j_{m}}\right\vert ^{p_{m}^{\ast}}\right)  ^{\frac{q_{m-1}}{p_{m}^{\ast}}%
}\right)  ^{\frac{q_{m-2}}{q_{m-1}}}\cdots\right)  ^{\frac{q_{1}}{q_{2}}%
}\right)  ^{\frac{1}{q_{1}}}\nonumber\\
&  \overset{\text{(\ref{900})}}{=}\left(  \sum_{j_{1}=1}^{\infty}\left(
\sum_{j_{2}=1}^{\infty}\cdots\left(  \sum_{j_{m-1}=1}^{\infty}\left(
\sum_{j_{m}=1}^{\infty}\left\vert \Psi^{-1}(A)(e_{j_{1}},\ldots,e_{j_{m}%
})\right\vert ^{p_{m}^{\ast}}\right)  ^{\frac{q_{m-1}}{p_{m}^{\ast}}}\right)
^{\frac{q_{m-2}}{q_{m-1}}}\cdots\right)  ^{\frac{q_{1}}{q_{2}}}\right)
^{\frac{1}{q_{1}}}\nonumber\\
&  \leq C_{p_{1},\ldots,p_{m}}\left\Vert \Psi^{-1}(A)\right\Vert \nonumber\\
&  \leq C_{p_{1},\ldots,p_{m}}\left\Vert A\right\Vert \nonumber
\end{align}
for all continuous $\left(  m-1\right)  $-linear operators $A:X_{p_{1}}%
\times\cdots\times X_{p_{m-1}}\rightarrow\left(  X_{p_{m}}\right)  ^{\ast}$.
Since $\left(  X_{p_{m}}\right)  ^{\ast}$ has cotype $\max\{p_{m}^{\ast},2\},$
by Theorem \ref{cotcrit}, the exponents $q_{1},\ldots,q_{m-1}$ in
(\ref{cotcrit}) satisfy%
\begin{equation}
q_{1}\geq\lambda_{\max\{p_{m}^{\ast},2\}}^{p_{1},...,p_{m-1}},q_{2}\geq
\lambda_{\max\{p_{m}^{\ast},2\}}^{p_{2},...,p_{m-1}},\ldots,q_{m-1}\geq
\lambda_{\max\{p_{m}^{\ast},2\}}^{p_{m-1}}. \label{oto}%
\end{equation}
Since
\[
1-\frac{1}{\max\{p_{m}^{\ast},2\}}=\frac{1}{\mu}%
\]
we have
\begin{align*}
\lambda_{\max\{p_{m}^{\ast},2\}}^{p_{i},...,p_{m-1}}  &  =\frac{1}%
{\max\left\{  \frac{1}{\max\{p_{m}^{\ast},2\}}-\left(  \frac{1}{p_{i}}%
+\cdots+\frac{1}{p_{m-1}}\right)  ,0\right\}  }\\
&  =\frac{1}{\max\left\{  1-\left(  \frac{1}{p_{i}}+\cdots+\frac{1}{p_{m-1}%
}+\frac{1}{\mu}\right)  ,0\right\}  }\\
&  =\delta^{p_{i},...,p_{m-1},\mu}%
\end{align*}
for all $i\in\left\{  1,...,m-1\right\}  $. Then, (\ref{oto}) can be re-stated
as%
\[
q_{1}\geq\delta^{p_{1},...,p_{m-1},\mu},q_{2}\geq\delta^{p_{2},...,p_{m-1}%
,\mu},\ldots,q_{m-1}\geq\delta^{p_{m-1},\mu}%
\]
and the proof is done.

(2)$\Rightarrow$(1). If the exponents $q_{1},...,q_{m-1}$ satisfy%
\[
q_{1}\geq\delta^{p_{1},...,p_{m-1},\mu},q_{2}\geq\delta^{p_{2},...,p_{m-1}%
,\mu},\ldots,q_{m-1}\geq\delta^{p_{m-1},\mu},
\]
we have, again, that the exponents $q_{1},\ldots,q_{m-1}$ satisfy
\[
q_{1}\geq\lambda_{r}^{p_{1},...,p_{m-1}},q_{2}\geq\lambda_{r}^{p_{2}%
,...,p_{m-1}},\ldots,q_{m-1}\geq\lambda_{r}^{p_{m-1}},
\]
with $r=\cot\left(  X_{p_{m}}\right)  ^{\ast}$. Thus, by Theorem
\ref{cotcrit}, there is a constant $C_{p_{1},...,p_{m-1}}^{\left(  X_{p_{m}%
}\right)  ^{\ast}}\geq1$ such that%
\[
\left(  \sum_{j_{1}=1}^{\infty}\left(  \sum_{j_{2}=1}^{\infty}\cdots\left(
\sum_{j_{m-1}=1}^{\infty}\left\Vert T(e_{j_{1}},\ldots,e_{j_{m-1}})\right\Vert
_{\left(  X_{p_{m}}\right)  ^{\ast}}^{q_{m-1}}\right)  ^{\frac{q_{m-2}%
}{q_{m-1}}}\cdots\right)  ^{\frac{q_{1}}{q_{2}}}\right)  ^{\frac{1}{q_{1}}%
}\leq C_{p_{1},...,p_{m-1}}^{\left(  X_{p_{m}}\right)  ^{\ast}}\left\Vert
T\right\Vert
\]
for all continuous $m$-linear operators $T:X_{p_{1}}\times\cdots\times
X_{p_{m-1}}\rightarrow\left(  X_{p_{m}}\right)  ^{\ast}.$

We thus have%
\begin{align*}
&  \left(  \sum_{j_{1}=1}^{\infty}\left(  \sum_{j_{2}=1}^{\infty}\cdots\left(
\sum_{j_{m-1}=1}^{\infty}\left(  \sum_{j_{m}=1}^{\infty}\left\vert A(e_{j_{1}%
},\ldots,e_{j_{m}})\right\vert ^{p_{m}^{\ast}}\right)  ^{\frac{q_{m-1}}%
{p_{m}^{\ast}}}\right)  ^{\frac{q_{m-2}}{q_{m-1}}}\cdots\right)  ^{\frac
{q_{1}}{q_{2}}}\right)  ^{\frac{1}{q_{1}}}\\
&  =\left(  \sum_{j_{1}=1}^{\infty}\left(  \sum_{j_{2}=1}^{\infty}%
\cdots\left(  \sum_{j_{m-1}=1}^{\infty}\left\Vert \Psi(A)(e_{j_{1}}%
,\ldots,e_{j_{m-1}})\right\Vert _{\left(  X_{p_{m}}\right)  ^{\ast}}^{q_{m-1}%
}\right)  ^{\frac{q_{m-2}}{q_{m-1}}}\cdots\right)  ^{\frac{q_{1}}{q_{2}}%
}\right)  ^{\frac{1}{q_{1}}}\\
&  \leq C_{p_{1},...,p_{m-1}}^{\left(  X_{p_{m}}\right)  ^{\ast}}\left\Vert
\Psi(A)\right\Vert \\
&  =C_{p_{1},...,p_{m-1}}^{\left(  X_{p_{m}}\right)  ^{\ast}}\left\Vert
A\right\Vert
\end{align*}
for all continuous $m$-linear forms $A:X_{p_{1}}\times\cdots\times X_{p_{m}%
}\rightarrow\mathbb{K}$.

\begin{remark}
The determination of the exact values of the constants involved in our main
theorem is probably a difficult task, as it happens with the Hardy--Littlewood
inequalities (see \cite{araujo, a22} and the references therein). However when
we are restricted to the bilinear case, with $p_{1}=p_{2}=\infty$ and $\sigma$
as the identity map, it is not difficult to check that we recover the constant
$\sqrt{2}$ from the Orlicz inequality.
\end{remark}


\bigskip

\begin{thebibliography}{99}                                                                                               %


\bibitem {abps2}N. Albuquerque, F. Bayart, D. Pellegrino, J. B.
Seoane-Sep\'{u}lveda, Optimal Hardy-Littlewood type inequalities for
polynomials and multilinear operators, Israel J. Math \textbf{211} (2016), 197--220.



\bibitem {rezende}N. Albuquerque, L. Rezende, Anisotropic regularity principle
in sequence spaces and applications. Comm. Contemp. Math. \textbf{20} (2018),
no. 7, 1750087, 14pp.

\bibitem {araujo}G. Ara\'{u}jo, K. C\^{a}mara, Universal bounds for the
Hardy-Littlewood inequalities on multilinear forms. Results Math. \textbf{73}
(2018), no. 3, Paper No. 124, 10 pp.

\bibitem {a22}G. Ara\'{u}jo, D. Pellegrino, D.\ Diniz P. da Silva e Silva, On
the upper bounds for the constants of the Hardy-Littlewood inequality. J.
Funct. Anal. \textbf{267} (2014), no. 6, 1878--1888.

\bibitem {laaaron}R. Aron, D. N\'{u}\~{n}ez-Alarc\'{o}n, D. Pellegrino and D.
M. Serrano-Rodr\'{\i}guez, Optimal exponents for Hardy--Littlewood
inequalities for $m$-linear operators, Linear Algebra Appl. \textbf{531
}(2017)\textbf{, }399--422.

\bibitem {bayart5}F. Bayart, Multiple summing maps: coordinatewise
summability, inclusion theorems and $p$-Sidon sets. J. Funct. Anal.
\textbf{274} (2018), no. 4, 1129--1154.

\bibitem {blasco}O. Blasco, G. Botelho, D. Pellegrino, P. Rueda, Summability
of multilinear mappings: Littlewood, Orlicz and beyond. Monatsh. Math.
\textbf{163} (2011), no. 2, 131--147.

\bibitem {blei}R. Blei, Analysis in Integer and Fractional Dimensions,
Cambridge Studies in Advances Mathematics. \textbf{71} Cambridge University
Press, Cambridge, 2001.

\bibitem {Di}J. Diestel, H. Jarchow, A. Tonge, Absolutely summing operators,
Cambridge Univ. Press, Cambridge, 1995.

\bibitem {hardy}G. Hardy, J. E. Littlewood, Bilinear forms bounded in space
$[p,q]$, Quart. J. Math. \textbf{5} (1934), 241--254.

\bibitem {mastilo}M. Masty\l {}o, An operator ideal generated by Orlicz
spaces. Math. Ann. \textbf{376} (2020), no. 3-4, 1675--1703.

\bibitem {mastilo2}M. Masty\l {}o, P. Mleczko, Norm estimates for matrix
operators between Banach spaces. Linear Algebra Appl. \textbf{438} (2013), no.
3, 986--1001.



\bibitem {tonge}B. Osikiewicz, A. Tonge, An interpolation approach to
Hardy--Littlewood inequalities for norms of operators on sequence spaces,
Linear Algebra Appl. \textbf{331} (2001), 1--9.

\bibitem {paulino}D. Paulino, Critical Hardy--Littlewood inequality for
multilinear forms. Rend. Circ. Mat. Palermo, II. Ser (2019). https://doi.org/10.1007/s12215-019-00409-0.

\bibitem {psst}D. Pellegrino, J. Santos, D. M. Serrano-Rodr\'{\i}guez, E. V.
Teixeira, Regularity Principle in Sequence Spaces and Applications. Bull. Sci.
Math. \textbf{141} (2017), 802--837.

\bibitem {pra}T. Praciano-Pereira, On bounded multilinear forms on a class of
$\ell_{p}$ spaces. J. Math. Anal. Appl. \textbf{81} (1981), 561--568.

\bibitem {rueda}P. Rueda, E. A, S\'{a}nchez-P\'{e}rez, A. Tallab, Traced
tensor norms and multiple summing multilinear operators. Linear Multilinear
Algebra \textbf{65} (2017), no. 4, 768--786.

\bibitem {sv}J. Santos, T. Velanga, On the Bohnenblust-Hille inequality for
multilinear forms. Results Math. \textbf{72} (2017), no. 1-2, 239--244.
\end{thebibliography}
\end{document}